\documentclass[12pt]{article}

\newcommand{\notes}[1]{}
\newcommand{\vekk}[1]{}
\newcommand{\delete}[1]{}
\newcommand{\gtcomment}[1]{}

\usepackage[colorlinks=true,linkcolor=blue,citecolor=blue]{hyperref} 
\usepackage{doi}

\usepackage{natbib}

\usepackage{color} % \gtcomment

\usepackage[latin1]{inputenc} % æ,ø, å
\usepackage{latexsym}

\usepackage{amsmath}
\usepackage{amsfonts}
\usepackage{amssymb}

\usepackage{amsthm}
\usepackage{float}

\usepackage[letterpaper]{geometry} % US standard

\newtheorem{theo}{Theorem}%Slett ikke sikkert at....
\newtheorem{Def}{Definition}

\newtheorem{lemma}{Lemma} %Slett ikke sikkert at....
 % ... 

\newcommand{\ceps}{{\cal E}} % a sigma field
\newcommand{\pr}{\operatorname{\text{P}}} % Probability
\newcommand{\Cov}{\operatorname{Cov}} % Probability

\newcommand{\abs}[1]{\left| #1 \right|} 
\newcommand{\st}{\operatornamewithlimits{{\mid}}}
\newcommand{\Normal}{\mathop{\bf Normal}} % Normal, or Gaussian
\newcommand{\gamvar}{\mathop{\bf Gamma}}
\newcommand{\invgamvar}{\mathop{\bf InvGamma}}

\newcommand{\univar}{\mathop{\bf U}}

\newcommand{\E}{\mathord{\mbox{$E$}}} % Expectation
\newcommand{\nc}{\newcommand}
\nc{\beqns}{\begin{eqnarray*}}
\nc{\eeqns}{\end{eqnarray*}}
\nc{\beqn}{\begin{eqnarray}}
\nc{\eeqn}{\end{eqnarray}}
\nc{\beq}{\begin{equation}}
\nc{\eeq}{\end{equation}}

\newcommand{\bes}{ \begin{equation} \begin{split} } %Nummerert slutt.
\newcommand{\ees}{ \end{split} \end{equation} } %Nummerert slutt.

\newcommand{\into}{\mbox{$\: \rightarrow \:$}}

\newcommand{\RealN}{{\mbox{$\Bbb R$}}} % Real numbers

\begin{document}

\newcommand{\change}[1]{{\color{red} #1}}
%\newcommand{\change}[1]{{#1}}

%\newcommand{\vecpsi}

% For
%\input{extra.aux}

\vekk{
% Comm stat
\setlength{\oddsidemargin}{0in}
\setlength{\evensidemargin}{0in}
\setlength{\topmargin}{-.5in}
\setlength{\headsep}{0in}
\setlength{\textwidth}{6.5in}
\setlength{\textheight}{8.5in}
\def\refhg{\hangindent=20pt\hangafter=1}
\def\refmark{\par\vskip 2mm\noindent\refhg}
\def\refhg{\hangindent=20pt\hangafter=1}
\def\refmark{\par\vskip 2mm\noindent\refhg}
\def\refhg{\hangindent=20pt\hangafter=1}    %20pt
\def\refhgb{\hangindent=10pt\hangafter=1}
\def\refmark{\par\vskip 2mm\noindent\refhg}
\renewcommand{\baselinestretch}{1.5}
}

%\title[Conditional sampling]{Conditional sampling and fiducial
%posteriors}
%\title{The Link between Fiducial and Posterior Sampling}
\title{Fiducial and Posterior Sampling}
% \title{On conditional, fiducial, and posterior sampling}

\author{Gunnar Taraldsen and Bo H. Lindqvist (2015)\\
  Communications in Statistics - Theory and Methods,\\
  44: 3754-3767,
  \doi{10.1080/03610926.2013.823207}}
  %  DOI:10.1080/03610926.2014.935430}

\date{}

\maketitle

%\newpage

\vekk{
\vfill

\begin{center}
\textbf{Abstract}
\end{center}
}

\begin{abstract}
The fiducial coincides with the posterior in a group model equipped with the right Haar prior. This result is here generalized. For this the underlying probability space of Kolmogorov is replaced by a $\sigma$-finite measure space and fiducial theory is presented within this frame. Examples are presented that demonstrate that this also gives good alternatives to existing Bayesian sampling methods. It is proved that the results provided here for fiducial models imply that the theory of invariant measures for groups cannot be generalized directly to loops: There exist a smooth one-dimensional loop where an invariant measure does not exist.
\end{abstract}

% It is well known that the fiducial distribution coincides with 
% the posterior in a group model equipped with the right Haar prior. 
% This result is here generalized within a theoretical frame that
% includes improper laws.
% For this the  underlying probability space of Kolmogorov is replaced by 
% a $\sigma$-finite measure space.
% The concept of a fiducial model and fiducial inference is
% also presented within this frame.
% It is proved more generally that a fiducial model can be used
% as a tool to obtain a Bayesian posterior sampling method.
% Examples 
% - multivariate normal, gamma, and others - 
% are presented that demonstrate that this
% can be a good alternative to existing Bayesian methods.
% A loop is an algebraic structure similar to a group,
% but the associative law is not assumed to hold.
% It is proved that the results provided here 
% for fiducial models imply that the theory of
% invariant measures for groups can not be generalized directly to loops:
% There exist a smooth one-dimensional loop where an invariant measure 
% does not exist.   

% conditional probability space as defined by Renyi. 
% The method of proof also provides
% a method for conditional sampling that generalizes previous results of the authors
% on conditional sampling given a sufficient statistic. 
% Both results provide tools 
% for the construction of novel inference methods.
% A method for conditional sampling.
%The gamma distribution is well known to be most useful in many 
%applied problems.
\vspace*{.3in}

% 3-6 keywords

\newpage

\tableofcontents

%\vfill

\newpage

%\doublespacing

\section{Introduction}

A Bayesian posterior is said to be a
fiducial posterior if it coincides with a fiducial distribution.
The question of existence of a Bayesian prior such that the
resulting posterior is a fiducial posterior has attracted 
interest since the introduction of the fiducial 
argument by \citet{Fisher30,Fisher35Fiducial}.
Cases where the fiducial is not a Bayesian posterior are interesting
because the fiducial theory then brings truly new armory
for the construction of new inference procedures.
The cases where there is a fiducial posterior are interesting because
the corresponding fiducial algorithm can be simpler to 
implement than the competitors based on the Bayesian theory.

The best result in the one-dimensional case was found by
\citet{Lindley58FiducialBayesEq}. 
He proved that, given appropriate smoothness conditions, 
a fiducial posterior exists if and only if the
problem can be transformed by one-one transformations of
the parameter and sample space into the standard location problem.
The best result obtained so far in the multivariate case was found by
\citet{Fraser61FiducialInvariance,Fraser61Fiducial}.
For a group model where a right Haar measure exists,
the fiducial coincides with the posterior 
from the right Haar measure as a Bayesian prior.
The main result in this paper is Theorem~\ref{TheoFidPost}
% on page~\pageref{TheoFidPost}
that contains both results as special cases.

\section{Fiducial posteriors}
\label{sFidPost}

The arguments in the following 
make it necessary to include improper priors in the
considerations, and this will here be done based 
on the theory presented by \citet{TaraldsenLindqvist10ImproperPriors}.
A brief summary of the necessary ingredients from this
theory is given next.

\begin{Def}[The basic space]
\label{defOmega}
The basic space $\Omega$ is equipped with a 
$\sigma$-finite measure $\pr$ defined on the
$\sigma$-field $\ceps$ of events.
\end{Def}

All definitions and results in the following will implicitly or explicitly 
rely on the existence of the underlying basic space.
This is as in the theory of probability presented by
\citet{KOLMOGOROV}, 
but the requirement $\pr (\Omega) = 1$ is here replaced by the
weaker requirement that $\pr$ is $\sigma$-finite:
There exist events $A_1, A_2, \ldots$ with 
$\Omega = \cup_i A_i$ and $\pr (A_i) < \infty$.
The above can be summarized by saying that it is assumed throughout
in this paper that
the basic space $\Omega$ is a $\sigma$-finite measure space 
$(\Omega, \ceps, \pr)$.

\begin{Def}[Random element]
\label{defZ}
A random element in a measurable
space $(\Omega_Z, \ceps_Z)$ is given by
a measurable function
$Z: \Omega \into \Omega_Z$.
The law $\pr_Z$ of $Z$ is defined by
\begin{equation}
\pr_Z (A) = \pr \circ Z^{-1} (A) = \pr (Z \in A) = \pr \{\omega \st Z (\omega) \in A\}.
\end{equation}
The random element is $\sigma$-finite if the law is $\sigma$-finite.
\end{Def}
Definition~\ref{defZ} corresponds to the
definition $\pr_X = \pr \circ X^{-1}$ by \citet[p.4,eq.2]{LAMPERTI} and
the definition $\mu_X = \mu \circ X^{-1}$ by
\citet[p.607]{SCHERVISH}.
It also corresponds to the original definition
given by \citet[p.21, eq.1]{KOLMOGOROV},
but he used superscript notation instead of the above subscript notation. 
The law $\pr_Z$ will also be referred to as the distribution of $Z$.
The term {\em random quantity} is used by \citet{SCHERVISH} and 
can be used as an alternative to
the term {\em random element} used above and by \citet{Frechet48ranElement}.
The term {\em random variable} $X$ is reserved
for the case of a random real number.
This is given by a measurable 
$X: \Omega \into \Omega_X = \RealN$,
where $\ceps_X$ is the $\sigma$-field generated by the
open intervals.
% The choice of a subscript notation is consistent with
% the subscript notation 
% $F_X (x) = \pr (X \le x) = \pr \{\omega \st X
% (\omega) \le x\}$ used for the
% definition of the cumulative distribution function

Definition~\ref{defZ} of a random element is more general than
any of the above given references since 
$(\Omega, \ceps, \pr)$ is not required to be a probability space,
but it is assumed to be a $\sigma$-finite measure space.
The space $\Omega$ comes, however, equipped with a large family of 
conditional distributions that are true probability distributions.
This is exactly what is needed for the formulation of 
a statistical inference problem,
and will be explained next.

Let $X$ and $Y$ be random elements,
and assume that $Y$ is $\sigma$-finite.
Existence of the conditional expectation
$\E (\phi (X) \st Y=y) = \E_X^y (\phi) = \E_X (\phi \st Y=y)$ and
the factorization
\begin{equation}
\pr_{X,Y} (dx,dy) = \pr_X^y (dx) \pr_Y (dy)
\end{equation}
can then be established.
The proof follows from the Radon-Nikodym theorem
exactly as in the case where the underlying space
is a probability space 
\citep{TaraldsenLindqvist10ImproperPriors}.
The case $X(\omega)=\omega$ gives in particular
$\{(\Omega, \ceps, \pr^y) \st y \in \Omega_Y\}$ as a family
of probability spaces.
This last claim is not strictly true,
but given appropriate regularity conditions there will exist
a regular conditional law as claimed \citep[p.618]{SCHERVISH}.
 
A {\em statistical model} is usually defined to be a 
family $\{(\Omega_X, \ceps_X, \pr_X^\theta) \st \theta \in \Omega_\Theta \}$
of probability spaces.
% indexed by the model parameter space $\Omega_\Theta$.
This definition is also used here,
but with an added assumption included in the definition:
It is assumed that there exist a random element $X$,
and a $\sigma$-finite random element $\Theta$
so that $\pr_X^\theta (A) = \pr_X (A \st \Theta = \theta)$.
It is in particular assumed that
both the sample space $(\Omega_X, \ceps_X)$ and the model parameter space
$(\Omega_\Theta, \ceps_\Theta)$ are measurable spaces.
The law $\pr_\Theta$ is not assumed to be known and
is not specified.
Similarly, the functions $X: \Omega \into \Omega_X$ and
$\Theta: \Omega \into \Omega_\Theta$ are assumed to exist,
but they are also not specified.
This is by necessity since the
underlying space $\Omega$ is not specified.
It is an abstract underlying space that makes it possible to formulate
a consistent theory.

A {\em Bayesian model} is given by a statistical model and the
additional specification of the law $\pr_\Theta$ of $\Theta$.
This prior law $\pr_\Theta$ can be improper in the
theory as just described,
and discussed in more detail by
\citet{TaraldsenLindqvist10ImproperPriors}.
The posterior law $\pr_\Theta^x$ is well defined if $X$ is $\sigma$-finite.
The result of Bayesian inference is given by the posterior law,
and Bayesian inference is hence trivial except for the
practical difficulties involved in the calculation of the posterior
and derived statistics.
The most difficult part from a theoretical perspective 
is to justify the choice of statistical model
and the prior in concrete modeling cases. 

Fiducial arguments were invented by \citet{Fisher30,Fisher35Fiducial}
to tackle cases without a prior law,
but with the aim to obtain a result similar to the posterior distribution.
The resulting distribution from the fiducial argument is called a
fiducial distribution.
The following definition \citep{TaraldsenLindqvist13fidopt} 
will be used here. 
It should be noted that the definition uses concepts that rely
on existence of the underlying basic space $\Omega$.
\begin{Def}[Fiducial model and distribution]
\label{defFidMod}
Let $\Theta$ be a $\sigma$-finite random element
in the model parameter space $\Omega_\Theta$.
A fiducial model $(U, \zeta)$ is defined
by a random element $U$ in the Monte Carlo space $\Omega_U$ 
and a measurable function 
$\zeta: \Omega_U \times \Omega_\Theta \into \Omega_Z$
where $\Omega_Z$ is the sample space.
The model is conventional if the conditional law
$\pr_U^\theta$ does not depend on $\theta$.
The model is simple if the fiducial equation
$\zeta (u,\theta) = z$ has a unique solution
$\theta^z (u)$ for all $u,z$.
If the model is both conventional and simple,
then the fiducial distribution 
corresponding to an observation $z \in \Omega_Z$
is the distribution of 
$\Theta^z = \theta^z (U)$ where 
$U \sim \pr_U^\theta$.
\end{Def}
A fiducial model $(U, \zeta)$ is
a fiducial model for the 
statistical model $\{\pr_Z^\theta \st \theta \in \Omega_\Theta\}$ if
\begin{equation}
(\zeta (U,\Theta) \st \Theta = \theta) \sim (Z \st \Theta = \theta)
\end{equation}
The fiducial model gives a method for simulation 
from the statistical model:
If $u$ is a sample from the known Monte Carlo law $\pr_U^\theta$,
then $z = \zeta (u,\theta)$ is a sample from $\pr_Z^\theta$.
Sampling from the fiducial follows likewise,
but involves solving the fiducial equation $\tau (u,\theta)=t$
to obtain the sample $\theta = \theta^t (u)$.
This, and related definitions in the literature, 
are discussed in more detail by \citet{TaraldsenLindqvist13fidopt}.

We have now presented the necessary ingredients for the formulation
of the main theoretical results here.
The first result gives conditions that ensure that
the fiducial coincides with a Bayesian posterior. 
\begin{theo}
\label{TheoFidPost}
Assume that $(U,\tau)$ is a conventional simple fiducial
model for the statistical model 
$\{\pr_T^\theta \st \theta \in \Omega_\Theta\}$.
If the Bayesian prior $\pr_\Theta$ implies that distribution of
$\tau (u, \Theta)$ does not depend on $u$,
then the Bayesian posterior distribution 
$\pr_\Theta^t$
is well defined and identical with the fiducial distribution of
$\Theta^t$.
\end{theo}
It should in particular be observed that the
required $\sigma$-finiteness of $T = \tau (U,\Theta)$
is a part of the conclusion in the previous theorem.
This ensures that the Bayesian posterior exists. 

The next result gives a recipe for posterior sampling
based on a fiducial model.
\begin{theo}
\label{TheoPost}
Assume that $(U,\tau)$ is a conventional fiducial
model and that $T = \tau (U,\Theta)$ is $\sigma$-finite
for a given prior $\pr_\Theta$.
Assume furthermore that
$\tau (u, \Theta) \sim w (t,u) \mu (dt)$   
for some $\sigma$-finite measure $\mu$ and
jointly measurable $w$.
If $u$ is a sample from a probability distribution
proportional to $w (t,u) \pr_U^\theta (du)$
and $\theta$ is a sample from the conditional law
$(\Theta \st \tau (u, \Theta) = t)$,
then $\theta$ is a sample
from the Bayesian posterior 
distribution of $\Theta$ given $T = t$.
\end{theo}

The proofs of Theorem~\ref{TheoFidPost} and
Theorem~\ref{TheoPost}
are postponed until 
Section~\ref{sLemma}.
We choose to discuss examples and consequences of these results next.

%\subsection{The location problem}
\section{The location problem}
\label{sLocation}

Assume that $t$ is the observed realization
of a random variable where
\begin{equation}
\label{eqLocFid}
t = \tau(u,\theta) = u + \theta
\end{equation}
where $u$ is a sample from the conditional law $\pr_U^\theta$.
It is assumed that  $\pr_U^\theta$
is known and does not depend on $\theta$.
The pair $(U, \tau)$ is then a fiducial
model for the statistical model
$\pr_T^\theta$.
The problem is to make statistical inference
regarding the model parameter 
$\theta \in \Omega_\Theta = \RealN$ based on
the model and the observation $t \in \Omega_T = \RealN$.

Consider first fiducial inference.
The fiducial distribution is determined
by the solution $\theta^t (u) = t - u$
of the fiducial equation $t = u + \theta$.
Monte Carlo sampling $u$ from the known law
$\pr_U^\theta$ gives corresponding samples
$t - u$ from the fiducial distribution.
The mean and standard deviation can then be calculated with
a precision depending on the choice of Monte Carlo size
and the random number generator. 
This can then be reported as an estimate of $\theta$
and a standard error respectively.
A more complete report can be given by a direct Monte Carlo
estimate of the fiducial distribution itself in the form of a graph.
This represents then the state of knowledge regarding
$\theta$ based on the observation and the fiducial model.

Consider next Bayesian inference.
Assume for simplicity that
$\pr_U^\theta (du) = f (u) \, du$,
where $du$ is Lebesgue measure on the real line.
If $\pr_\Theta (d\theta) = \pi (\theta) \, d\theta$
is the prior law,
then the posterior law is given by a density 
$\pi (\theta \st t) = C_t f (t - \theta) \pi (\theta)$
where $C_t$ is a normalization constant.
This normalization is generally possible if 
$T = \tau (U,\Theta)$ is $\sigma$-finite,
and this happens exactly when
$\int f (t - \theta) \pi (\theta)\,d\theta < \infty$ for (almost) all $t$ \citep{TaraldsenLindqvist10ImproperPriors}.
It is always possible if $\pi$ is a probability density,
but an alternative sufficient condition is that
$\pi$ is bounded.
The particular case $\pi (\theta) = 1$ gives the
result $\pi (\theta \st t) = f (t - \theta)$.
A simple calculus exercise shows directly
that this coincides with the fiducial law derived above.
The reporting of the result can be done as in the case of 
fiducial inference.

The previous result can also be inferred from Theorem~\ref{TheoFidPost}
since the law of $\tau (u, \Theta) = u + \Theta$ is the
Lebesgue measure when the law of $\Theta$ is the Lebesgue measure.
More generally the assumption $\pr_\Theta (d\theta) = \pi (\theta) \, d\theta$
gives $\tau (u, \Theta) \sim \pi (t-u) \, dt$.
Theorem~\ref{TheoPost} and the assumption of
$\sigma$-finiteness of $T$ 
can then be used for Bayesian
sampling more generally as follows:
Sample $u$ from a law proportional to the
measure $\pi (t - u) \, \pr_U^\theta (du)$ and
return $\theta = t - u$.
The latter proof does not rely on the existence of a density
$f$ for $\pr_U^\theta$ with respect to Lebesgue measure.
This argument is used in Section~\ref{ssSingular}
to provide a concrete example where a traditional
Bayesian sampling recipe fails,
but the fiducial algorithm from
Theorem~\ref{TheoPost} can be used.

Consider finally frequentist inference.
If $\pr_U^\theta (du) = f (u) \, du$
where $f$ has a unique maximum at $m_u$,
then $\hat{\theta} = t - m_u$ is the maximum likelihood estimator.
Assume that the expected value
$\E^\theta U = \mu_u$ exists.
It follows then that the expected value 
$\bar{\theta} = t - \mu_u$
of the fiducial distribution is
the shift equivariant estimator with smallest
mean square error \citep{TaraldsenLindqvist13fidopt}.
It is in particular better than the maximum likelihood estimator
when both exist,
it is unbiased, 
and the standard error is given by the standard deviation
of $U$.

The fiducial distribution is also a confidence distribution
since $U = T - \theta$ is a pivotal.
Consequently an expanded uncertainty can be found corresponding
to $95\%$ confidence intervals.
Symmetric, shortest, or uniformly most powerful limits can be
calculated.
The most powerful limits follow with reference to the likelihood ratio test
as exemplified for the exponential by \citet{Taraldsen11exprounded}.
This reference also gives the route for the inclusion
of the effect of finite resolution into the analysis.

The previous analysis with the assumption 
$\Omega_U = \Omega_T = \Omega_\Theta = \RealN$
can be generalized verbatim to the case
$\Omega_U = \Omega_T = \Omega_\Theta = V$
where $V$ is a finite dimensional real or complex vector space.
The property $\tau (u, \Theta) = u + \Theta \sim \Theta$
holds for the finite dimensional Lebesgue distribution for $\Theta$.
The further generalization to the case where $V$
is an infinite dimensional Hilbert space gives an example
where the Bayesian algorithm fails to produce 
optimal frequentist inference.
The fiducial argument given above holds also for the infinite
dimensional case,
and gives optimal inference as stated above
\citep{TaraldsenLindqvist13fidopt}.

The analysis can be generalized further to the case 
$\Omega_U = \Omega_T = \Omega_\Theta = V^n$.
This includes in particular the case of a random sample of size
$n$ from the original model given in
equation~(\ref{eqLocFid}),
but the independence assumption is not required in the following argument.
Equation~(\ref{eqLocFid}) must be replaced by
the equation $T_1 = U_1 + \theta$ corresponding to the
first component of the random element $T$ in $V^n$.
The law $\pr_U^\theta$ must be replaced by
the conditional law 
$(U_1 \st \Theta=\theta,U_2-U_1=t_2-t_1,\ldots,U_n-U_1=t_n-t_1)$.
Except for the practical difficulties 
related to this conditional law,
the analysis proceeds as before.
Optimal frequentist inference procedures 
including confidence distributions follow from the resulting
fiducial also in this case 
\citep{TaraldsenLindqvist13fidopt}.

\section{Location examples}
\label{sLocExamples}

% The aim next is to presents examples 
% to illustrate the general theory.
% The first examples are elementary, 
% and the resulting inference
% procedures are not novel.
% The novelty lies in the theoretical 
% formulation based on the underlying basic space $\Omega$. 

\subsection{A singular example}
\label{ssSingular}

The purpose of this example is to demonstrate that 
Theorem~\ref{TheoPost} can be used to calculate the
Bayesian posterior in certain cases where the 
traditional Bayesian recipe fails.

Let the Monte Carlo law $\pr_U^\theta$ give probability $p_i$ to the value $u_i$
for $i=1,2$. 
The model given by
$X = U + \theta$ with $\theta \in \Omega_\Theta = \RealN$
gives a law $\pr_X^\theta$ which is concentrated on
$\{u_1+\theta, u_2+\theta\} \subset \Omega_X = \RealN$.
The traditional Bayesian posterior would usually be calculated
by $\pi (\theta \st x) \propto f (x \st \theta) \pi (\theta)$,
but this fails here since the density $f (x \st \theta)$ fails to
exist for the case considered.

Consider next the algorithm given by Theorem~\ref{TheoPost}.
The relation 
$u + \Theta \sim \pi (x - u) \, dx$ gives the following
recipe:
Sample $u$ from a law that gives
relative probability $q_i = \pi (x - u_i) p_i$ to the values $u_1$ and $u_2$. 
The resulting
$\theta = x - u$ is a sample from the Bayesian
posterior $\pr_\Theta^x$.
The conclusion is that the posterior 
gives probability $q_i /(q_1 + q_2)$ to the two
values $\theta_i = x - u_i$ for $i=1,2$.

The uniform prior case 
$\pi = 1$ gives that the posterior equals the
fiducial which gives probability $p_i$ to $\theta_i$.

\subsection{Normal distribution}
\label{ssNormal1}

Assume that $X_i = \chi_i (U, \theta) = \theta + \sigma_0 U_i$
where the Monte Carlo law of $U$ corresponds to a random sample
of size $n$ from the standard normal distribution.
The $(U, \chi)$ is then a fiducial model for a random sample of
size $n$ from
a $\Normal (\theta, \sigma_0^2)$
where the variance $\sigma_0^2$ is assumed known.

This gives
$T = \overline{X} = \overline{\chi (U, \theta)} = \tau (V, \theta) =
\theta + \sigma V$ with
$\sigma = \sigma_0 / \sqrt{n}$ and
$V = \sqrt{n} \overline{U}$ has a standard normal distribution.
The $(V, \tau)$ is then a fiducial model for 
the sufficient statistic $T$ which is
a $\Normal (\theta, \sigma^2)$ statistical model.

The fiducial based on the sufficient statistic is the law of
$\Theta^t = t - \sigma V$ which is 
$\Normal (t, \sigma^2)$.
This gives the optimal equivariant estimator
$t=\overline{x}$,
the standard error $\sigma=\sigma_0/\sqrt{n}$,
and the expanded error $k \sigma$
where the coverage factor $k=1.96$ gives 
the level $95\%$.

The Bayesian conclusion with the uniform law as prior 
is given by the same numbers since the fiducial coincides with the
posterior in this case.

\subsection{Gamma distribution I}
\label{ssGammaI}

The example here is a generalization of the case given by a random sample
from the exponential distribution \citep{TaraldsenLindqvist13fidopt}.
Let 
$\chi_i (u, \theta) = \theta F^{-1} (u_i; \alpha)$,
where $F^{-1}$ is the inverse CDF of the gamma distribution 
with scale $\beta = 1$.
If $(U_1, \ldots, U_n \st \Theta=\theta) \sim \univar (0,1)$ independent, 
then the inversion method gives 
that $(U, \chi)$ is
a fiducial model for a random sample $X=(X_1, \ldots, X_n)$
from the gamma density:
\begin{equation}
f_{X_i} (x_i \st \theta) = 
\{\theta^\alpha \Gamma (\alpha)\}^{-1}
\; {x_i}^{\alpha - 1} e^{-x_i/\theta}
,\;\;
\text{shape } \alpha > 0, \text{scale } \theta > 0
\end{equation}  
It follows from this density that 
$T=\overline{X}$ is sufficient.
A fiducial model from the above fiducial model
is then
$T = \theta V$,
where the Monte Carlo variable
$V = \overline{F^{-1} (U; \alpha)}$
has a $\gamvar (n \alpha,1/n)$ distribution.

The fiducial equation 
$t=\theta v$ gives the fiducial 
$\Theta^t = t/V$ with an
$\invgamvar (n \alpha, n t)$ distribution.
This is a confidence distribution for $\theta$,
and also the Bayesian posterior corresponding
to a uniform prior for $\log \theta$.
The mean 
\begin{equation}
\overline{\theta} = t /(\alpha - 1/n)
\end{equation}  
is the best Bayesian estimator for the quadratic loss.
It can be seen as a sample size
adjustment of the likelihood estimate $\hat{\theta}=t/\alpha$.

The scale model transforms to the
location model
$\ln t = \ln (\theta) + \ln (v)$.
The best equivariant estimator for $\ln (\theta)$ is then
$\E^\theta (\ln \Theta^t)$,
and this integral equals $\ln (n t) - \psi (n \alpha)$ where
$\psi$ is the digamma function.

The best equivariant estimate for $\theta$ is
\begin{equation}
\label{eqBestBetaScale}
\tilde{\theta}=  (t/\alpha) \exp(\ln (n \alpha) - \psi (n \alpha))
\end{equation}  
This is best with respect to the squared distance 
$\abs{\ln (\theta_1) - \ln (\theta_2)}^2$ from
the Fisher metric as explained in more detail by 
\citet{TaraldsenLindqvist13fidopt}.

The reason for the choice of the above formulation of
equation~(\ref{eqBestBetaScale}) is that
$t$ is the uniformly minimum variance estimator of
$\alpha \beta$,
and the $\exp(\cdot)$ term can be seen as a correction of this.
The following
asymptotic and divergent series
$\psi (x) - \ln (x) \sim 1/(2x) - \sum_{n \ge 1} B_{2n} / (2 n x^{2n})$
for $x \rightarrow \infty$
shows in particular consistency of the estimator
in equation~(\ref{eqBestBetaScale}) with the more common estimator
$t/\alpha$ in the limit of infinite sample size $n \rightarrow \infty$.

The main reason for the inclusion of this example
is not the possibly novel result given by equation~(\ref{eqBestBetaScale}),
but rather demonstration purposes.
We consider the arguments as given above as a competitive alternative 
to the arguments given by a Bayesian calculation.
The Bayesian calculation is of course possible in this case,
but it seems more cumbersome to us.  
The claims on optimality can indeed also be proved directly
without any mention of a Bayesian prior \citep{TaraldsenLindqvist13fidopt}.

\section{One-dimensional fiducial inference}
\label{ssLindley}

Fiducial inference was first considered in the one-dimensional case.
This is discussed here, and the connection between the original 
definition and the more general Definition~\ref{defFidMod}
is in particular explained.

\subsection{Lindley's result}

\citet{Lindley58FiducialBayesEq} considered the one-dimensional case.
His sufficient and necessary condition for a fiducial posterior 
is equivalent with the conditions given in Theorem~\ref{TheoFidPost}.
The result is only valid by consideration of a more restrictive
definition of the fiducial distribution defined directly and uniquely
by an absolutely continuous cumulative distribution function.
This is explained next.

The monotonicity of a simple fiducial model has as a consequence
monotonicity in $\theta$ of the cumulative distribution function 
$F (t \st \theta)$ of $\pr_T^\theta$.
In the following it is furthermore assumed
that $\theta \mapsto F(t \st \theta)$ is
absolutely continuous and onto $(0,1)$.
The relation $u = \hat{u} (t, \theta) = F (t \st \theta)$ can be
inverted to give $\theta = \hat{\theta} (u, t)$ and
$t = \tau (u,\theta) = F^{-1} (u \st \theta)$.
The well known inversion method gives that
$\{U, \tau \}$ with $\pr_U^\theta$ the uniform law on $(0,1)$
is a fiducial model for $\pr_T^\theta$.
This is the Fisher fiducial model,
and it is a simple conventional fiducial model.
%and pivotal.\label{pFishFid}
It can be shown that the corresponding Fisher fiducial distribution coincides with
the fiducial distribution of the original fiducial model \citep{DawidStone82}.
Fiducial inference is hence unique in this case.
If the cumulative distribution $F(t \st \theta)$ is decreasing in $\theta$,
then $1 - F(t \st \theta)$ is the cumulative fiducial distribution.

% This will be assumed in the following.
% Many other cases can be reduced to this case,
% but this will not be discussed further here.

The result of Lindley is
that a fiducial posterior is obtained
if and only if the fiducial model $(U, \tau)$ is a transformation
of a fiducial model $(V, \eta)$ where
$\eta (v, \varsigma) = \varsigma + v$.
The prior for $\varsigma$ is Lebesgue measure on $\RealN$ and the
resulting fiducial model $(V, \eta)$ is the location model.
The transformation assumption is that 
$\tau (u,\theta) = \phi_3 (\phi_1 (u) + \phi_2 (\theta))$ with 
$v = \phi_1 (u)$ and $\varsigma = \phi_2 (\theta)$.

The if part of the claim is a special case of the results
discussed in Section~\ref{sLocation} since both Bayesian and fiducial
inference behave consistently under transformations.
The if part does not require existence of densities,
and this result here is then an extension of the results of Lindley.

An example which is more general than the Lindley case is obtained by choosing a
$\phi_3$ which is strictly increasing, but nowhere differentiable.
The result is then a singular continuous fiducial posterior,
and this is not covered by the proof of Lindley.
Another class of examples is given by choosing an
arbitrary probability distribution $\pr_U^\theta$ which does not need to be 
absolutely continuous.
A third class of examples not covered by Lindley is given by countable
$\Omega_U = \Omega_\Theta = \Omega_T \subset \RealN$ 
equipped with a possibly non-commutative group or loop operation.

It remains to prove the only if part of the fiducial posterior claim
given the above restrictions on the cumulative distribution.
The necessary parts of the argument of Lindley is reproduced next.

Assume that the fiducial model has a fiducial posterior
in the sense that the fiducial density
$-\partial_\theta F (t \st \theta)$ 
equals the posterior density 
$\partial_t F (t \st \theta) h(\theta) g(t)^{-1}$.
This gives the following generalization of the one-way wave equation
\begin{equation}
\label{eqWave}
-h(\theta)^{-1}\partial_\theta F (t \st \theta)=
g(t)^{-1} \partial_t F (t \st \theta) 
\end{equation}
A general solution is given by 
$F = S(G(t)-H(\theta))$, where
$G'=g$ and $H'=h$.
Consequently, the family of conditional distributions for $G(T)$ is 
a location family with location parameter $H(\theta)$.
The one-one correspondence 
$G(T) \mapsto G^{-1} (G(T)) = T$ proves that $T$
is given by a transformation of the location group model.

A particularly nice aspect of the above proof % of \citet{Lindley58FiducialBayesEq}
is that it gives explicitly the required transformation to a
standard location model.
The function $G$ is the cumulative distribution 
of the (marginal) law of $T$, 
and a fiducial posterior is obtained if and only if
the variable $G(T)$ corresponds to a standard location model.

It is also the explicit transformation that ensures that
the law of $H(\Theta)$ is the uniform law on the real line:
\begin{equation}
\pr (a < H(\Theta) < b) = \int_{H^{-1} (a)}^{H^{-1} (b)} h(\theta)\,d\theta = b-a
\end{equation}
A particular consequence is that the prior law of $\Theta$ is always
improper when the posterior coincides with the fiducial.

\subsection{The correlation coefficient}
\label{ssCorr}

Let $F (r \st \rho)$ 
be the cumulative distribution function of
the empirical correlation coefficient $r$
of a random sample from the bivariate normal distribution.
The parameter $\theta=\rho$ is the correlation coefficient.
A fiducial model $(U,\chi)$ is given by
a uniform law $\pr_U^\theta$ and the
fiducial relation $\chi (u,\theta) = F^{-1} (u \st \theta)$.
It is possible to sample from the fiducial based on $F$,
but a much simpler algorithm is described in section~\ref{ssMNormal}
below.

In this case it is known that
there exists no prior on $\rho$ 
that gives the fiducial as a posterior $\pr_\rho^r$ 
\citep[p.966]{BergerSun08}.
The proof is not trivial.
The fiducial for the correlation coefficient
gives the very first example \citep{Fisher30} of a derivation
of a fiducial distribution \citep[p.176]{FISHER}.
The fiducial for the correlation coefficient is, 
however, 
a Bayesian posterior from the multivariate normal model
considered in section~\ref{ssMNormal}.

\subsection{Gamma distribution II}
\label{ssGammaII}

Consider a random sample from the
gamma density
\begin{equation}
f_{X_i} (x_i \st \theta) = 
\{\beta^\theta \Gamma (\theta)\}^{-1}
\; {x_i}^{\theta - 1} e^{-x_i/\beta}
,\;\;
\text{shape } \theta > 0, \text{scale } \beta > 0
\end{equation}  
The case with a general scale $\beta$
can be reduced to the case of a scale $\beta=1$
by consideration of $x_i/\beta$.
It will hence initially be assumed that $\beta=1$.

The form of the density shows that
$T=\overline{\ln (X)}$ is sufficient.
A fiducial model is given by
$T = \tau (U,\theta) = \overline{ \ln (F^{-1} (U; \theta))}$
where $F^{-1}$ is the inverse CDF of the gamma distribution 
with scale $\beta = 1$ and
$(U_1, \ldots, U_n \st \Theta=\theta) \sim \univar (0,1)$ independent, 
Each $\ln (F^{-1} (u_i; \theta))$ is increasing in $\theta$,
since $F (u \st \theta)$ is increasing.

An alternative fiducial model is given by
$T = G^{-1}(V, \theta)$ where $G(t \st \theta)$ is the
CDF of $T$ and $(V \st \Theta=\theta) \sim \univar(0,1)$.
An explicit expression for $G$ can be given
in terms of the Meijer G-function using results by \citet{Nadarajah11gammaprod}.
Both models are simple, and give the same fiducial distribution.
We conjecture that this fiducial is not obtainable as
a Bayesian posterior, 
but do not attempt a proof.

The fiducial is a confidence distribution for the shape $\theta$,
and both of the previous fiducial models give sampling algorithm.
Reasonable estimators for $\theta$ are given by
$\E^{t,\theta} \Theta^x$ and $\exp \left( \E^{t,\theta} \log(\Theta^x) \right)$
corresponding to a squared distance loss on the direct and logarithmic
scale respectively.
Alternatives are given by the Fisher information metric or an entropy distance.
Natural competitors are the maximum-likelihood and the Jeffreys prior Bayesian
versions of the previous fiducial estimators.
A detailed discussion of this will not be give here.

% \subsection{The truncated exponential}

\section{Group and loop models}
\label{sGroup}
%\label{ssFraser}

It will next be explained, as promised in the abstract, 
that Fraser's result on fiducial posteriors 
follows as a special case of Theorem~\ref{TheoFidPost}.
%It is assumed throughout this section that
%$\pr_U^\theta$ does not depend on $\theta$. 

\subsection{A generalized location-scale model}
\label{ssLocScale}

Let $\chi$ be defined by
\begin{equation}
\label{eqFidNorm}
x_i = \chi_i (u, \theta) = \theta u_i = [\mu, L] u_i = \mu + L u_i,
\;\; i=1,\ldots,n
\end{equation}
where $\mu, u_i \in \RealN^p$ are columns of length $p$ and 
$L$ is a lower triangular $p \times p$
matrix with positive diagonal.
The case $p=1$ gives the standard location-scale model
$x_i = \mu + \sigma u_i$ with $\sigma=L$,
and equation~(\ref{eqFidNorm}) can be seen as a natural generalization.

The generalized location-scale group $G = \Omega_\Theta$ with elements
$\theta = (\mu, L)$ is discussed in more detail by
\citet[p.175]{Fraser79inferenceAndLinear}
in the context of structural inference.
Multiplication is defined by
$[\mu_1, L_1] [\mu_2, L_2] = [\mu_1 + L_1 \mu_2, L_1 L_2]$,
the inverse is $[\mu, l]^{-1} = [-l^{-1} \mu, l^{-1}]$,
and the identity is $e = [0, I]$.
The group may also be identified with the group of lower triangular matrices
on the $2\times 2$ block form 
\begin{equation}
\label{eqGroupNorm}
g = 
\begin{pmatrix}
L & 0_p\\
\mu^* & 1
\end{pmatrix}
\end{equation}
which gives the previous calculation rules from matrix multiplication directly.

A Monte Carlo law $\pr_U^\theta$ gives that
$(U, \chi)$ is a fiducial model for the
conditional law $\pr_X^\theta$ of $X = \chi (U, \Theta)$.
It will be assumed that the $U_i$ are independent
and corresponds to a random sample of size $n$ from
a  known probability distribution on $V=\RealN^p$.
The columns $X_i = \mu + L U_i$
corresponds then also to a random sample of size $n$
from a distribution on $V$ with
$\E^\theta (X_i) = \mu$ and
$\Cov^\theta X_i = \E^\theta (X_i-\mu)\,(X_i-\mu)^* = L\,L^* = \Sigma$,
where it is assumed that $\Cov^\theta U_i = I$ 
and $\E^\theta U_i = 0$.

The result so far is a fiducial model where the model parameter space
corresponds to the mean $\mu$ and covariance $\Sigma$ of
some multivariate law on $V=\RealN^p$.
It should be observed that the Cholesky decomposition
$\Sigma = L\,L^*$ determines $L$ uniquely,
and it is hence a matter of choice if $L$ or $\Sigma$ is considered
as a model parameter.

The model given by equation~(\ref{eqFidNorm}) 
is not a simple fiducial model,
but it can be reduced to a simple fiducial model by
conditioning similarly to how the location model was treated.
The general recipe for this is explained by
\cite{TaraldsenLindqvist13fidopt},
but the details of this will not be give here.

\subsection{The multivariate normal}
\label{ssMNormal}

The possibly most important group model is 
given from the previous discussion and
assuming that 
$\pr_U^\theta$ is the law of a $p \times n$ matrix of
independent standard normal variables.
The result is then a fiducial model $(U, \chi)$ corresponding to a
random sample of size $n$ from the multivariate
normal $\Normal_p (\mu, \Sigma)$.
This is not the only possible fiducial model for this case,
but other possibilities will not be discussed.

A simple fiducial model is then obtained from the
sufficient statistic $T=(\overline{X}, L_x)$
where $L_x L_x^*$ is the Cholesky decomposition of the 
empirical covariance matrix of $X$.
The fiducial model $(U, \chi)$ from equation~(\ref{eqFidNorm})
gives then a fiducial model $(V, \tau)$
for $\pr_T^\theta$ where
\begin{equation}
\label{eqFidSNorm}
t = \tau (u, \theta) = \theta v,\;\; t,\theta,v \in G
\end{equation}
and $v = [\overline{u}, L_u]$.
This model is simple,
and the fiducial as given by
Definition~\ref{defFidMod}
is the law of
$\Theta^t = t V^{-1}$.

Let $\pr_\Theta$ be the right Haar prior on $G$.
The explicit form for this is not needed
in the following argument.
The right invariance gives $\tau (v, \Theta) = \Theta v \sim \Theta$,
and Theorem~\ref{TheoFidPost} gives
that the Bayesian posterior coincides with
the fiducial.

Sampling from the posterior can be done by alternative methods,
but it seems that the algorithm that follows from the fiducial
argument is the simplest possible that
generate independent samples.
It involves only standard matrix calculations including solving lower triangular
linear systems, 
and calculation of Cholesky decompositions.
This gives in particular a simple sampling algorithm for the
fiducial distribution of the correlation coefficient
considered in section~\ref{ssCorr}.

\subsection{General group case}
\label{ssGenGroup}

Assume that $t = \tau (v, \theta) = \theta v$
is given by group multiplication.
Let $\pr_\Theta$ be a right Haar measure on the group
$G = \Omega_\Theta = \Omega_V = \Omega_T$.
It follows from the right invariance that 
$\tau (u, \Theta) = \Theta v \sim \pr_\Theta$ for all $v$.
Theorem~\ref{TheoFidPost} can now be applied,
and it follows that the distribution of
$tV^{-1}$ conditional on $\Theta = \theta$ equals both the fiducial and the posterior.
This case can be referred to as the {\em fiducial group model} case \label{pFidGroup}. 
The fiducial model for $T$ is pivotal,
and the right Haar distribution is a matching prior:
The posterior is fiducial and also a confidence distribution
since $v = \theta^{-1} t$ gives a pivotal quantity.
This result is the result obtained by 
\citet{Fraser61FiducialInvariance},
but he obtained it by a different argument.

The previous group case is important since it gives
a multitude of non-trivial examples where the assumptions
in Theorem~\ref{TheoFidPost} are fulfilled.
It is in particular noteworthy that the required 
$\sigma$-finiteness of $T$ and $(U,T)$ follows as consequences
in the fiducial group model.  

Existence of a $\sigma$-finite random quantity $\Theta$ such that
the distribution $\tau (u,\Theta)$ does not depend on $u$
is a non-trivial problem in general.
It is a generalization of the existence and uniqueness problem 
for Haar measure on a group.
This is ensured in the fiducial group case if it is assumed that $G$ is a
locally compact group \citep{HALMOS}.

A more general family of examples can be constructed as follows.
Let 
$\phi_3: G \into \Omega_T$,
$\phi_2: \Omega_\Theta \into G$,
$\phi_1: \Omega_U \into G$,
and 
$\tau(u,\theta) = \phi_3(\phi_2 (\theta) \phi_1 (u))$
where the product is the group multiplication in 
a group $G$ equipped with a right Haar measure $\mu$.
Assume that $\phi_2$ is such that
$\phi_2 (\Theta) \sim \mu$.
It follows then that $\tau(u,\Theta) \sim \tau (U,\Theta) \sim \mu_{\phi_3}$.
If $\phi_3$ and $\phi_2$ are invertible, 
then the fiducial model is simple and the
fiducial posterior is distributed like
$\phi_2^{-1} (\phi_3^{-1}(t) (\phi_1 (U))^{-1})$ 
conditional on $\Theta = \theta$.
%with $U \sim \pr_U^\theta$.
It can be observed that the functions 
$\phi_2$ and $\phi_3$
can be used
to identify $G$ and $\Omega_\Theta$ 
respectively 
$\phi_3 (G) \subset \Omega_T$
both as sets and as groups.
The previous model is hence essentially reducible to the
group case by a change of variables.

The reduction can alternatively be formulated as follows.
Define a new parameter $\varsigma = \phi_2 (\theta)$,
a new variable $V = \phi_1 (U)$,
and let $\eta (v, \varsigma) = \varsigma v$.
It follows that $(V, \eta)$ is a fiducial group model for 
$S = \phi_3^{-1} (T)$, 
and inference can be based on this.
The result is the same as in the previous paragraph.
It can not be concluded that the original
fiducial model $(U, \tau)$ is a group model,
but the model is transformed into a fiducial 
group model $(V, \eta)$. \label{pLoopGroupMod}

Let $G$ be a group with an invariant measure,
and let $\tau (u,\theta) = \phi_3 (\phi_2 (\theta) \phi_1 (u))$
with $\phi_i$ one-one on $G$.
This is a special case of the case considered in the previous two paragraphs. 
This defines a binary operation on $G$ which 
need not be a group since the associative law may fail.
An example is given by $\tau (u,\theta) = (u + \theta)/2$ with addition on
the real line $G=\RealN$.
It is however a quasi-group \citep{Smith06loop}, 
but in the context here it is
essentially reduced to the group case by 
relabeling as explained in the previous paragraphs.
% The loop is then isotopic to a group,
% and there exist general theorems that 
% characterize the loops that are isotopic to a group

\subsection{Loop models}

A {\em quasigroup} $(G,\circ)$ is a set $G$ equipped with a binary
operation $\circ$ such that for each $a, b \in G$, 
there exist unique elements $x,y \in G$ such that
$a\circ x = b$ and $y \circ a = b$.
A {\em loop} $(G,\circ, e)$ is a quasigroup with an identity element,
that is, an element $e$ such that $a \circ e = a = e \circ a$ for all
$a \in G$.
A {\em group} $(G,\circ, e)$ is a loop so that the associative law
$(a \circ b) \circ c = a \circ (b \circ c)$ holds for all
$a,b,c \in G$.
The concept of a loop within abstract algebra as just defined is 
probably less familiar to most readers than the concept of a group.
On an intuitive level it can be considered to be 
an object similar to a group,
but without the associative law.
It will next be explained that loops occur naturally
in the context of fiducial theory.

Consider the case $t = \tau (u, \theta)$ where 
$\tau$ is a bijection separately in both arguments.
The fiducial model is then said to be a  
{\em pivotal} and {\em simple} model.
The bijections defined by $\tau$
can be used to define a change of variables so
that it may be assumed that $\Omega_T = \Omega_U = \Omega_\Theta = G$.
The result is a set $G$ equipped with a binary operation
$\tau (u, \theta)$ with inverse $\hat{\theta} (u,t)$ and
inverse $\hat{u} (t,\theta)$. 
$G$ is then a quasi-group.
The notation $\tau (u, \theta) = \theta u$,
$\hat{\theta} (u,t) = t/u$, 
and $\hat{u} (t,\theta) = \theta \setminus t$ with right and left
division is standard.
The change of variables can also be chosen so that there is an
identity element $e$ such that
$g e = e g = g$ for all $g$.
$G$ is then a loop.
The conclusion is that a change of variables reduces a pivotal simple
model to a loop model.\label{pLoop}
% A loop can be characterized intuitively as a group without the associative law.   

Two examples of loops which 
do not seem to be essentially 
reduced to the group case are given next. 
One example is given by
$\tau (u, \theta) = (\theta u^2 \theta)^{\frac{1}{2}}$
where $u$ and $\theta$ are positive definite matrices.
This gives an example of a Bruck loop.
Another example is $\tau (u, \theta) = \theta u$ where
the multiplication is the multiplication of the invertible octonions,
which is a Moufang loop. 
It is not known to the authors if there exist invariant 
loop measures for certain classes of loops, or for these concrete examples.
This would provide examples beyond the group case.  

The finite and countable loop cases are trivial in 
that counting measure
is the unique invariant measure,
but they can otherwise be quite exotic objects.
They do, however, provide examples where the fiducial equals the
Bayesian posterior,
and this does not follow from the results of
Lindley and Fraser.

The one-dimensional case considered in Section~\ref{ssCorr} provides in particular an
example of a loop on the real line that does not possess an invariant measure.
This could be of independent interest, and is hence stated separately
here as a Theorem.
\begin{theo}
\label{TheoNoInvariant}
There exist a smooth loop where an invariant right measure does not exist.
\end{theo}
\begin{proof}
This follows from the correlation coefficient example in
section~\ref{ssCorr} which provides a loop model where the fiducial
can not be a Bayesian posterior.
An invariant right measure would provide a Bayesian posterior
as a special case of Theorem~\ref{TheoFidPost}.
\end{proof}

The result of Lindley can be reformulated to give an alternative
characterization of loops with an invariant measure:
A smooth loop on the real line has an invariant measure if and only if it 
can be reduced to a group by a change of variables. 
It is unclear if this can be generalized to more general loops.
The term {\em smooth} is here interpreted to mean infinitely
differentiable with continuous derivatives.

\vekk{
\subsection{Gamma distribution III}
\label{ssGammaIII}

2 d example of loop not reducible to a group.
}

\section{A fundamental lemma}
\label{sLemma}

The following Lemma has Theorems~\ref{TheoFidPost}-\ref{TheoPost}
as direct consequences.
\begin{lemma}
\label{Lemma1}
Assume that $\Theta$ is $\sigma$-finite with
$(U \st \Theta = \theta) \sim f(u) \nu (du)$ and
$\tau (u, \Theta) \sim w (t, u) \mu (dt)$ 
for fixed $u$, where $\nu$ and $\mu$ are $\sigma$-finite measures
and $w$ is jointly measurable.
Let $T=\tau (U,\Theta)$.
It follows then that
\begin{equation}
%h(u,t) = f (u \st \hat{\theta} (u, t)) w (t, u) 
(U,T) \sim  f (u) w (t, u)\,  \nu (du) \mu (dt)
\end{equation}
\end{lemma}
\begin{proof}
Change variables from $(U, \Theta)$ to $(U, T)$
%using $\theta = \hat{\theta} (u, \tau (u, \theta))$:
%
\begin{equation*}
\begin{split}
\E \phi(U,T) & =
\iint \phi (u, \tau(u,\theta)) f(u) \, \nu (du) \pr_\Theta (d\theta)\\
& =
\iint \phi (u, t) f(u) w(t, u)
\, \mu (dt) \, \nu (du)
\end{split}
\end{equation*}

\end{proof}

The key ingredients in the above proof are the Fubini theorem together
with the general change-of-variables theorem 
$\int \psi (z) \pr_Z (dz) = \int \psi (\phi(y)) \pr_Y (dy)$
when $Z = \phi (Y)$.
This theorem is usually proved in the context
of probability spaces \citep[p.163, Theorem C]{HALMOS},
but the proof is also valid for the more general case where $\pr$ is
assumed to be $\sigma$-finite.
%The proof actually holds also without the assumption of

The main point of Lemma~\ref{Lemma1} is that it provides an explicit expression
for the density $h$ of $(U,T)$,
and it follows in particular that this density exists.
It follows from the proof
that $(U,T)$ is $\sigma$-finite,
and that $T$ is $\sigma$-finite if and only if
$\int h(u,t)\; \nu (du) < \infty$ for $\mu$-a.e. $t$.
This condition can be checked in applications.

\begin{proof}
{\bf Theorem~\ref{TheoFidPost}}
The assumption gives $\tau (u, \Theta) \sim \mu (dt)$ for
a $\sigma$-finite measure $\mu$.
The $\sigma$-finiteness follows since
$\theta \mapsto \tau(u,\theta)$ is a bijection. 
Lemma~\ref{Lemma1} gives that 
$h (u,t) = f(u)$ is the joint density of $(U,T)$,
and then also that $T \sim \mu (dt)$ is $\sigma$-finite.
A sample $\theta$ from
$\Theta \st T=t$ can generally be obtained
by sampling $u$ from $(U \st T=t)$ followed
by sampling $\theta$ from
$(\Theta \st T=t, U=u)$.
The result is identical with the result $\theta^t (u)$ 
from the fiducial as defined in Definition~\ref{defFidMod}
since $h(u,t) = f(u)$ is the density of $(U \st T=t)$.
\end{proof}

\begin{proof}
{\bf Theorem~\ref{TheoPost}}
The proof is as the previous,
but the density of $(U \st T=t)$
is now given by a density
proportional to $h (u,t) =  f (u) w (t, u)$.
The required normalization is possible since it is assumed that
$T$ is $\sigma$-finite. 
\end{proof}

\section{Closing remarks}
\label{sClosing}

As explained in the introduction it is important
to establish cases where the fiducial equals a Bayesian posterior,
and also the cases where the fiducial is not a Bayesian posterior.
In general and special cases this is a difficult task.
Theorem~\ref{TheoFidPost} shows that
existence of a law $\pr_\Theta$ such that $\tau (u,\Theta)$
has a law that does not depend on $u$ implies
that the fiducial equals the resulting Bayesian posterior.
The proof of existence of an invariant law $\pr_\Theta$ is difficult,
but in the case of groups the theory is well established.
The results of Fraser on fiducial posteriors follow then as
corollaries of Theorem~\ref{TheoFidPost}
as explained in section~\ref{ssGenGroup}.

The correlation coefficient case gives an example where
the fiducial is not equal to a Bayesian posterior  
from the statistical model for the empirical correlation coefficient.
We believe that the gamma with known scale, and the gamma where both
scale and shape are unknown give two more examples
where the fiducial is not a Bayesian posterior,
but we do not have a proof of this.
 
The use of sufficient statistics for the gamma model gives
examples of respectively a one- and a two-dimensional loop model.
The question of existence of invariant measures for quasi-groups or
loops has here been shown to be related to the question of
fiducial posteriors.
Unfortunately, it seems that the question of existence of invariant
measures for quasi-groups is an open and difficult question.
A byproduct of the discussion given here is
Theorem~\ref{TheoNoInvariant} that shows existence of 
a smooth loop where an invariant right measure does not exist.

Theorem~\ref{TheoPost} has a more direct application.
It gives an alternative algorithm for Bayesian posterior
sampling based on a fiducial model.

%\bibliography{bib}
\bibliography{bib,gtaralds}
%\bibliography{gtaralds}
%\bibliographystyle{asa}
%\bibliography{bib}
%\bibliographystyle{chicago} 
%\bibliography{gtaralds} 
% bibinputs! men virker ikke ??? c:/ d:/ problem ??
% Virket etter restart?
%\bibliographystyle{plain}
\bibliographystyle{plainnat} 

\end{document}